\documentclass[10pt,leqno,a4paper]{article}
\usepackage{amsmath}
\usepackage{amsthm}
\usepackage{amssymb}
\usepackage{graphicx}
\usepackage{epstopdf}
\usepackage{setspace}
\usepackage{hyperref}
\usepackage{pdfsync}
\usepackage{caption}
\usepackage{amsmath,amscd,amssymb}
\usepackage[english]{babel}
\usepackage[inner=2.5cm,outer=2.5cm,bottom=5cm]{geometry}
\usepackage{setspace}
\usepackage{mathrsfs}
\usepackage[all]{xy}
\usepackage{tikz-cd}
\usepackage{bbm}
\usepackage{titlesec}

\newtheorem{thm}{Theorem}[section]
\newtheorem{cor}[thm]{Corollary}
\newtheorem{lemma}[thm]{Lemma}

\newtheorem{defn}[thm]{Definition}

\theoremstyle{remark}

\theoremstyle{definition}

\newtheorem{rmk}[thm]{Remark}
\newtheorem{rmks}[thm]{Remarks}

\numberwithin{equation}{thm}

\def\beq{\begin{equation}}
\def\eeq{\end{equation}}

\def\ben{\begin{enumerate}}
\def\een{\end{enumerate}}

\def\crash#1{}
\def\N{{\mathbb N}}
\def\Z{{\mathbb Z}}

\def\R{{\mathbb R}}

\def\D{{\mathbb D}}

\def\l{\left}
\def\r{\right}
\def\[[{\l[\l[}
\def\]]{\r]\r]}

\def\ord{{\rm ord}\;}

\def\lc{\emph{loc.cit.}\;}

\def\cA{{\mathcal A}}

\def\cD{{\mathcal D}}
\def\cE{{\mathcal E}}

\def\cM{{\mathcal M}}

\def\cO{{\mathcal O}}

\def\cL{{\mathcal L}}

\def\cS{{\mathcal S}}
\def\cT{{\mathcal T}}

\def\cX{{\mathcal X}}

\def\vphi{\varphi}

\def\res{{\rm Res}}
\def\1{{\mathbbm{1}}}
\newcommand{\V}[1]{{\color{blue}#1}}

\makeatletter
\def\blfootnote{\xdef\@thefnmark{}\@footnotetext}
\makeatother

\begin{document}
\title{Mittag-Leffler problems on Berkovich curves}
\author{Velibor Bojkovi\'c}
\providecommand{\keywords}[1]{\textbf{\textit{Key words and phrases. }} #1}

\date{}

\maketitle
\begin{abstract}
Given a quasi-smooth Berkovich curve $X$ admitting a finite triangulation, finitely many disjoint open annuli $A_1,\dots,A_n$ in $X$ that are not precompact, and for each $i=1,\dots, n$, an analytic function $f_i$ (resp. differential form $\sigma_i$) convergent on $A_i$, we provide a criterion for when there exists an analytic function $f$ (resp. a differential form $\sigma$) on $X$ inducing the functions $f_i$ (resp. a differentials $\sigma_i$). 

Along the way we reprove residue theorem for differentials on smooth Berkovich curves that admit finite triangulations. 
\end{abstract}

\blfootnote{\keywords{Mittag-Leffler problem, Berkovich curves, $p$-adic Runge's theorem}}
\tableofcontents

\subsection*{Introduction}

If $\cS$ is a compact Riemann surface, $p_1,\dots,p_n$ a finite number of points on $\cS$, and for each point $p_i$, $f_i$ is a Laurent polynomial in a local parameter at $p_i$, \emph{Mittag-Leffler problem for meromorphic functions} asks whether there exists a meromorphic function $f$ on $\cS$, holomorphic outside points $p_i$ and whose principal part at $p_i$ coincides with $f_i$. A similar problem for differential forms $\sigma_i$ in place of functions $f_i$ is called \emph{Mittag-Leffler problem for meromorphic differentials} (see \cite[Chapter VI]{MirBook} and \cite[Sections 2.3. and 4.9.]{DanBook1adv}).  

Both problems played a prominent role in the development of the function theory on Riemann surfaces as well as of the cohomological methods which now days are indispensable tools for the study of Riemann surfaces or algebraic curves. As an instance of this, Mittag-Leffler problems are closely related to such cornerstone results such as Riemann-Roch theorem and Serre's duality as is accounted for example in \cite[Chapter VI Section 3.]{MirBook} (or any other book on Riemann surfaces, as a matter of fact). \\

In the present note we study similar problems for analogues of Riemann surfaces over \emph{nonarchimedean} fields. If $k$ is an algebraically closed, complete, nontrivially valued and non-archmedean field of characteristic 0 (as is the underlying assumption throughout the paper), the role of the ``Riemann surfaces over $k$'' will be taken by quasi-smooth $k$-analytic curves in the sense of Berkovich analytic geometry over $k$. More precisely, those curves $X$ which admit finite triangulations (hence we call them finite), which, loosely speaking means that we can take away finitely many points $\cT$ out of $X$ such that the remaining is a disjoint union of open discs and \emph{finitely many} open annuli (for example, analytification of smooth projective $k$-algebraic curves are finite, but so are quasi-smooth $k$-affinoid curves and many more). Out of these annuli, there will be some which are not precompact in $X$ and we call these \emph{boundary}. The problem that we study is (see Section \ref{s: M-L} for precise formulations): \emph{Given a finite $k$-analytic curve $X$ and finitely many open boundary annuli and on each of then an analytic function (resp. differential form) expressed as a series of the chosen local coordinate, does there exists a global analytic function (resp. differential form) on $X$ that induces these functions?} We note that the classic Mittag-Leffler problems for Riemann surfaces can be expressed in a similar fashion as soon as we agree to consider punctured discs to be open annuli (which we do later on; see also Remark \ref{rmk: classic M-L}).

Our answers, and main results of the present note, namely Theorems \ref{thm: M-L differenitals} and \ref{thm: M-L funcitons} strongly resemble criteria provided for the classical Mittag-Leffler problems. The main tools used in their proofs are 1) a residue theorem for differentials on smooth finite Berkovich curves (that are called wide open curves) and 2) an approximation argument that allows us, by modifying analytic functions and differential forms on our boundary annuli, to pass to the classical Mittag-Leffler problems for which we know the criteria. \\

In the first sections we recall some basic properties of the curves involved, and also get familiar with residues of differential forms on them. An approximation argument on which we base our results is nothing but the beautiful \emph{$p$-adic Runge's theorem} due to M. Raynaud  and which we also recall in this section.  Second section contains residue theorem (and an analogue of the ``inside-outside theorem'') for differentials on smooth finite curves which is certainly well known. However, the proof we provide, based also on the approximation argument seems to be new. Mittag-Leffler problems are discussed in the third section.

\section{Finite curves}
\subsection{Preliminaries}
Throughout the paper $k$ will be a complete, algebraically closed, nontrivially valued non-archimedean field of characteristic 0. The norm on $k$ will be denoted by $|\cdot|$ while the usual absolute value on $\R$ will be denoted by $|\cdot|_{\infty}$.
 
By a $k$-analytic curve we will mean a curve in the sense of Berkovich $k$-analytic geometry as developed in \cite{Ber90,BerCoh}. In particular, we will work with quasi-smooth $k$-analytic curves, which means that every type 1 point (using the Berkovich classification of points on $k$-analytic curves) has a neighborhood isomorphic to an open disc. If $X$ is a quasi-smooth $k$-analytic curve, then any locally finite subset $\cT$ of type 2 and 3 points such that $X\setminus \cT$ is a disjoint union of open discs and annuli\footnote{We consider punctured discs to be annuli as well, but also type 1 points to be closed discs of radius 0}, will be called a triangulation of $X$ (note a slight difference with the definition in \cite[Section 5.1.13]{Duc-book}). Then, Th\'eor\`eme 5.1.14 in \lc implies that every quasi-smooth $k$-analytic curve admits a triangulation. 

We say that $X$ is finite if it admits a finite triangulation. In this case, we have the following criterion for finite curves (\cite[Theorem 1.4.2.]{BojRH},\cite[Theorem 2.8.]{BojPoi-cc})
\begin{thm}
 A connected, quasi-smooth $k$-analytic curve $X$ is finite if and only if it is isomorphic to a complement of finitely many open and closed discs in a smooth projective $k$-analytic curve $X'$.
\end{thm}

\begin{defn} We will say that $X'$ from the theorem above is a {\em simple projectivization} of the curve $X$.  
\end{defn}

\begin{rmk}
 We will often identify $X$ with its image in $X'$ without explicitly mentioning the isomorphism involved.
\end{rmk}

We note the following particular classes of finite curves. Namely, the compact ones (they are finite since they admit a triangulation and in particular a finite one because of compactness). Then by \cite[Corollaire 6.1.4]{Duc-book} they are either $k$-affinoid either smooth projective $k$-analytic curves. A quasi-smooth $k$-affinoid curve can also be seen as a complement of finitely many open discs in a smooth projective $k$-analytic curve.

As opposed to compact ones, we also have $k$-analytic curves which are complements of finitely many closed discs in a smooth projective curve $X'$. We call such curves {\em wide open}\footnote{We note slight difference with the classic literature, where wide open curves are complements of finitely many closed discs of nonzero radius}. 

Finally, the remaining class of curves contains ones which are isomorphic to a complement of a nonzero (finite) number of open discs and a nonzero (finite) number of closed discs in a smooth projective $k$-analytic curve. We call them {\em semi-open}.

\begin{defn}
 Let $X$ be a finite curve. Then, any open annulus $A$ in $X$ which is not relatively compact and such that $X\setminus A$ is nonempty will be called a boundary annulus of $X$. 
 
 We denote the set of boundary annuli of $X$ as $\cE'(X)$. 
\end{defn}

In fact, if $X$ is finite, $X'$ its simple projectivization then each boundary annulus in $X$ has an endpoint in $X'$ which is the maximal point of one of the closed discs which are complements of $X$ in $X'$. 

We introduce the following the equivalence relation "$\sim$" on $\cE'(X)$ by saying that two annuli $A_1,A_2\in \cE(X)$ are in "$\sim$" if $A_1\subset A_2$ or $A_2\subset A_1$. Then, the set of classes $\cE'(X)/\sim$ is called the set of ends of $X$ and is denoted by $\cE(X)$.

Finally, if $X$ is a finite curve, $X'$ its projectivization, let $\cD$ be the set of open discs that are connected components of $X'\setminus X$. Then, we call the set of ends $\cE(D)$, where $D$ runs through $\cD$, {\em the outer ends of $X$} and we denote it by $\cE^o(X)$.

\subsection{Spectral norm} 

\subsubsection{} If $\cA_X$ is an affinoid algebra corresponding to the $k$affinoid curve $X$, we recall that $\cA_X$ is a $k$-Banach algebra for the sup-norm $|\cdot|_X$ defined as, for $f\in \cA_X$, $|f|_X:=\max_{x\in X}|f|_x$ (here and elsewhere $|\cdot |_x$ is the multiplicative seminorm in the multiplicative spectrum of $\cA_X$ inducing the point $x$), and it coincides with the spectral norm \cite[Theorem 1.3.1.]{Ber90}. By Corollary 2.4.5 in \lc the sup is actually achieved on the Shilov boundary of $X$,  $Sh(X)$. The later consists of finitely many points that can be characterized as follows. 

Let $X'$ be a simple projectivization of $X$ so that $X$ is isomorphic to (and in fact identified with) a complement in $X'$ of finitely many open discs, say $D_1,\dots,D_n$. Then, $Sh(X)$ consists of the points $\overline{D_i}\setminus D_i$, for $i=1,\dots,n$, and where $\overline{D_i}$ is the closure of $D_i$ in $X'$.

 Suppose now that $X$ is a strict quasi-smooth $k$-affinoid curve and $f\in \cA_X$. Let $Sh(X)=\{x_1,\dots,x_n\}$ and let further $A_1,\dots,A_n$ be finitely many open annuli in $X$ such that the closure of $A_i$ in $X$ intersected with $Sh(X)$ consists only of the point $x_i$, for $i=1,\dots,n$. For each $A_i$, let $t_i$ be a coordinate on $A_i$ identifying it with an open annulus $A(0;r_{i,1},r_{i,2})$ in the $t_i$-analytic affine line over $k$, and oriented in such a way that, if $\zeta_{i,r}$, $r\in (r_{i,1},r_{i,2})$ is the point on the skeleton of $A(0;r_{i,1},r_{i,2})$ of radius $r$, then $\lim\limits_{r\to r_2}\zeta_{i,r}\in Sh(X)$. Then, the restriction $f_{|A_i}$ of $f$ to $A_i$ can be expressed as a series $f_i(t_i)=\sum_{j\in\Z}f_{i,j}\,t^j_i$ which is convergent on $A(0;r_{i,1},r_{i,2})$.

By continuity of the function $|f|_{(\cdot)}:X\to \R$, $x\mapsto |f|_x$, we conclude 
\begin{equation}\label{eq: bd sup}
 |f|_X=\max_{1\leq i\leq n}\{\lim\limits_{r\to r_{i,2}}|f_i(t_i)|_{\zeta_{i,r}}\}, 
\end{equation}
and where $|f_i(t_i)|_{\zeta_{i,r}}=\max_{j}|f_{i,j}|r^j$. The equation \eqref{eq: bd sup} will come in handy on several occasions.


\subsubsection{}
More generally, suppose that $M$ is a finite module over an affinoid algebra $\cA_X$. Then, since $\cA_X$ is a Banach algebra, $M$ can be equipped with a norm with respect to which it becomes a Banach $\cA_X$-module, and any two such norms are equivalent \cite[Section 3.7.3]{BGR}. 

For example, if $X=A[0;r_1,r_2]$ is a closed annulus of inner radius $r_1$ and of outer radius $r_2$ equipped with a coordinate $t$, then one norm on $\Omega_X$ is given by $||\omega||=|\sum_{i\in \Z} \alpha_i t^i|_X$, where $\omega=\sum_{i\in\Z}\alpha_it^i\, dt$.

\subsection{$p$-adic Runge's theorem}
Let $X'$ be a smooth projective $k$-analytic curve , $\cX'$ the corresponding $k$-algebraic curve (that is, the analytification of $\cX'$ is $X'$) and let $X$ be a strict $k$-affinoid curve in $X'$. Let $U_1,\dots,U_n$ be the connected components of $X'\setminus X$, and for each $i=1,\dots,n$, let $c_i\in U_i(k)$ be a rational (type 1) point. Finally, let $\cX$ be the affine $k$-algebraic curve $\cX'\setminus\{c_1,\dots,c_n\}$. The following is $p$-adic Runge's theorem {\em a l\`a} M. Raynaud.

\begin{thm}\label{thm: runge} The ring of rational functions on $X'$ that have poles at most in the points $c_1,\dots,c_n$ (that is the coordinate ring of $\cX$) is dense in the ring $\cA_X$ with respect to the spectral norm. 

More generally, for any coherent sheaf $\cM$ on $\cX$ ($\cM$ induces a coherent sheaf on $X$), the restriction of the global sections $\cM(\cX)$ over $X$ is dense in the space of sections $\cM(X)$. 
\end{thm}
\begin{proof}
 For $X$ a strict $k$-affinoid curve this is essentially \cite[Corollaire 3.5.2]{Ray94}. Although, in the \lc the base field is discretely valued, the proof follows almost verbatim  to algebraically closed fields of characeristic 0 as in \cite[Theorem 6]{BojRH}\footnote{In \cite{BojRH} the field $k$ is assumed to be of mixed characteristic, but the proof carries on to the case of equal characteristic 0}.  
\end{proof}

\subsection{Residues}

\subsubsection{} Let $A$ be an open annulus. We recall that by our constructin of ends of a finite curve, $A$ has two ends, that is two classes in $\cE(A)$.


\begin{defn}
 An oriented open annulus $A$ is a pair $(A,e)$, where $e\in \cE(A)$. We may just write $A$ instead of $(A,e)$ if the end $e$ is understood from the context.
 
 We say that a finite morphism $f:A_1\to A_2$ of oriented open annuli $(A_1,e_1)$ and $(A_2,e_2)$ is orientation preserving if the class $e_1$ of $A_1$ is sent to the class $e_2$ of $A_2$. Otherwise, we say that $f$ is orientation reversing.
 
 Finally, an open annulus $A=A(0;r_1,r_2)$ will always be identified with an oriented annulus $(A,e)$, where $e$ is the end that contains $A(0;r_1,r)$, $r\in (r_1,r_2)$ in its class.
\end{defn}
For example, let $f:A(0;r_1,r_2)\to A(0;r_1',r_2')$  be a finite morphism of degree $d$ and let $t$ and $s$ be respective coordinates on the annuli. Then, $f$ can be written as $s=f(t)=\sum_{i\in \Z}f_i\, t^i$ and since $f$ has no zeroes, the theory of valuation polygons tells us there exists an $n\in \Z$ such that $f(t)=f_n\,t^n\,(1+h(t))$, where $|n|_{\infty}=d$, $h(t)$ has constant term $0$ and for every $r\in (r_1,r_2)$, $|h(t)|_r<1$. Then, $f$ is orientation preserving if $n=d$ and orientation reversing if $-n=d$.
\begin{rmk}
 We note that the composition of two orientation reversing morphisms is orientation preserving, and the composition of two orientation preserving morphisms is still orientation preserving. 
\end{rmk}

\subsubsection{}
\begin{lemma} (Compare with \cite[Lemma 2.1.]{Col89})
 Let $A$ be an oriented open annulus, let $t:A\to A(0;r_1,r_2)$ and $s:A\to A(0;r_1',r_2')$ be two orientation preserving coordinates, $\omega\in \Omega_A$ and let
 $$
 \omega=\sum_{i\in \Z}\alpha_i\, t^i\, dt=\sum_{i\in \Z}\beta_i\,s^i\, ds.
 $$
 Then, $\alpha_{-1}=\beta_{-1}$.
\end{lemma}
\begin{proof}
 We may write $s=f(t)=\sum_{i\in\Z}f_i\, t^i$, where $f$ is the corresponding isomorphism of open annuli $A(0;r_1,r_2)\to A(0;r_1',r_2')$. Moreover, since $f$ is in addition orientation preserving, we have 
 \begin{equation}\label{eq: iso relation}
  |f_1|\rho>|f_i|\rho^i,\quad i\neq 1,\quad \forall \rho\in (r_1,r_2).
  \end{equation}
Then,
\begin{align*}
\sum_{i\in\Z}\alpha_i\, t^i\, dt&=\sum_{i\in\Z}\beta_i\, s^i\, ds=\sum_{i\in\Z}\beta_i\, f(t)^i\, f'(t)\, dt\\
&=\beta_{-1}\frac{f'(t)}{f(t)}\,dt+\sum_{i\neq -1}\frac{1}{i+1}\big(f(t)^{i+1}\big)'\, dt,
\end{align*}
It follows that the only term contributing to $\alpha_{-1}$ in the last expression is $\beta_{-1}\frac{f'(t)}{f(t)}$, so it is enough to prove that the coefficient with $t^{-1}$ in $\frac{f'(t)}{f(t)}$ is 1. For an $N\in \N$ and $N>1$, let us put $f_N(t):=\sum_{i=-N}^Nf_i\, t^i$. Then, $\frac{f'(t)}{f(t)}=\lim_{N\to \infty}\frac{f'_N(t)}{f_N(t)}$ on the annulus $A(0;t_1,t_2)$. On the other hand
\begin{align}
\frac{f_N'(t)}{f_N(t)}&=\big(f_1+\sum_{\substack{i=-N\\i\neq 1}}^Ni\, f_i\, t^{i-1}\big)\,\big(f_1\, t+\sum_{\substack{i=-N\\i\neq 1}}^Nf_i\, t^i\big)^{-1}\nonumber\\
&=\frac{1}{t}\, \big(1+\sum_{\substack{i=-N\\i\neq 1}}^N\frac{i\,f_i}{f_1}t^{i-1}\big)\,\big(1+\sum_{\substack{i=-N\\i\neq 1}}^N\frac{f_i}{f_1}t^{i-1}\big)^{-1}\nonumber\\
&=\frac{1}{t}\, \big(1+\sum_{\substack{i=-N\\i\neq 1}}^Ni\,g_i\,t^{i-1}\big)\,\big(1+\sum_{j=1}^\infty(-1)^{j}\big(\sum_{\substack{i=-N\\i\neq 1}}^Ng_i\,t^{i-1}\big)^j\big),\nonumber
\end{align}
where we put $g_i:=f_i/f_1$ and the last expansion is valid all over $A(0;r_1,r_2)$ because of \eqref{eq: iso relation}. What remains is to prove that the constant term in the expression 
\begin{align}
 \big(1+\sum_{\substack{i=-N\\i\neq 1}}^Ni\,g_i\,t^{i-1}\big)\,\big(1+\sum_{j=1}^\infty(-1)^{j}\big(\sum_{\substack{i=-N\\i\neq 1}}^Ng_i\,t^{i-1}\big)^j\big)-1\nonumber\\
 =\sum_{\substack{i=-N\\i\neq 1}}^Ni\,g_i\,t^{i-1}\label{eq: expansion 1}\\
 +\sum_{j=1}^\infty(-1)^{j}\big(\sum_{\substack{i=-N\\i\neq 1}}^Ng_i\,t^{i-1}\big)^j\label{eq: expansion 2}\\
 +\big(\sum_{\substack{i=-N\\i\neq 1}}^Ni\,g_i\,t^{i-1}\big)\,\sum_{j=1}^\infty(-1)^{j}\big(\sum_{\substack{i=-N\\i\neq 1}}^Ng_i\,t^{i-1}\big)^j\label{eq: expansion 3}
\end{align}
is equal to 0. Obviously, the term \eqref{eq: expansion 1} does not contribute to the constant term of the sum. We next use the multinomial theorem applied to the sum \eqref{eq: expansion 2} and suppose that there exist some $J>1$, $j_1,\dots,j_l\geq 1$ and $i_1,\dots,i_l\in\{-N,\dots,N\}\setminus\{1\}$ such that 
\begin{align}
 j_1+\dots+j_l&=J,\nonumber\\
 j_1\,(i_1-1)+\dots+j_l\,(i_l-1)&=0. \label{eq: conditions}
\end{align}
Then (by using the multinomial theorem applied to the power $j=J$ in \eqref{eq: expansion 2}) we see that $(-1)^J\binom{J}{j_1,\dots,j_l}\,g_{i_1}^{j_1}\dots g_{i_l}^{j_l}$ contributes to the constant coefficient. On the other side, we see that the product of the coefficients $g_{i_1}^{j_1}\dots g_{i_l}^{j_l}$ can appear in the sum \eqref{eq: expansion 3} only in the term 
$$
\big(\sum_{\substack{i=-N\\i\neq 1}}^Ni\,g_i\,t^{i-1}\big)\,(-1)^{J-1}\big(\sum_{\substack{i=-N\\i\neq 1}}^Ng_i\,t^{i-1}\big)^{J-1}
$$ 
and more precisely in the following forms
\begin{align*}
(-1)^{J-1}\big(i_1\,g_{i_1}\,\binom{J-1}{j_1-1,j_2,\dots,j_l}g_{i_1}^{j_1-1}g_{i_2}^{j_2}\dots g_{i_l}^{j_l}+\dots+i_l\,g_{i_l}\,\binom{J-1}{j_1,j_2,\dots,j_l-1}g_{i_1}^{j_1}g_{i_2}^{j_2}\dots g_{i_l}^{j_l-1}\big)\\
=(-1)^{J-1}\big(i_1\,\binom{J-1}{j_1-1,j_2,\dots,j_l}+\dots+i_l\,\binom{J-1}{j_1,j_2,\dots,j_l-1}\big)\,g_{i_1}^{j_1}\dots g_{i_l}^{j_l}.
\end{align*}
The final touch comes by noticing that
\begin{align*}
&(-1)^{J-1}\big(i_1\,\binom{J-1}{j_1-1,j_2,\dots,j_l}+\dots+i_l\,\binom{J-1}{j_1,j_2,\dots,j_l-1}\big)\\
&=(-1)^{J-1}\big(\frac{i_1\,j_1}{J}+\dots+\frac{i_l\,j_l}{J}\big)\binom{J}{j_1,\dots,j_l}=(-1)^{J-1}\binom{J}{j_1,\dots,j_l},
\end{align*}
where the last equality holds because of \eqref{eq: conditions}. We conclude that all the terms in the sum \eqref{eq: expansion 2} that contribute to 0 cancel with the terms in the sum \eqref{eq: expansion 3} (and vice versa). By taking the limit $N\to \infty$ the lemma is proved.
\end{proof}
The previous lemma in fact proves that the following definition is good.
\begin{defn}
 Let $(A,e)$ be an oriented open annulus and $\omega\in \Omega_A$. We define the residue $\res_{(A,e)}(\omega)$ of $\omega$ at $e$ by putting
 $\res_{(A,e)}(\omega)=\alpha_{-1},$
 where $\omega=\sum_{i\in\Z}\alpha_i\, t^i\,dt$ and $t$ is any orientation preserving coordinate on $A$.
 
 If $A$ is clear from the context, we may just write $\res_e$ instead of $\res_{(A,e)}$.
 \end{defn}
 \begin{rmk}
  If $A$ is an open annulus which is isomorphic to the punctured open disc $D\setminus\{a\}$, then one of it ends $e$ can be identified with the point $a$. In this case, we will simply write $\res_a$ for $\res_e$.
 \end{rmk}

\begin{lemma}
 Let $A$ be an open annulus, $e_1$ and $e_2$ its two ends and $\omega \in \Omega_A$. Then
 $$
 \res_{e_1}(\omega)=-\res_{e_2}(\omega).
 $$
\end{lemma}
\begin{proof}
 Let $t:A\to A(0;r_1;r_2)$ be any orientation preserving coordinate on $(A,e_1)$ so that we can write $\omega=\sum_{i\in\Z}\alpha_i\, t^i\, dt$. Then $\res_{e_1}(\omega)=\alpha_{-1}$. To find $\res_{e_2}(\omega)$ it is enough to choose any orientation reversing coordinate on $(A,e_1)$ and in fact $s=\frac{1}{t}$ will do. The proof comes by rewriting $\omega$ in coordinate $s$ as follows
 $$
 \omega=\sum_{i\in\Z}\alpha_i\,\frac{1}{s^i}\,(-1)\frac{1}{s^2}\,ds.
 $$
\end{proof}

Now, let $W$ be a wide open curve and $\cE(W)$ its ends, $e\in \cE(W)$ and  $\omega\in\Omega_W$. Let further $A_1$ and $A_2$ be any open annuli in $W$ which belong to the end $e$. Then, naturally $e$ can be identified with an element in $\cE(A_1)$ and $\cE(A_2)$.

\begin{lemma}\label{lem: residue oposites}
 We have
 $$
 \res_{(A_1,e)}(\omega)=\res_{(A_2,e)}(\omega).
 $$
\end{lemma}
\begin{proof}
 Since both $A_1$ and $A_2$ belong to the same end, we may suppose that $A_1\subset A_2$.  Then, $A_1$ is also an element of the end $e$ considered as the end of $A_2$. Hence, if $t$ is an orientation preserving coordinate on $(A_2,e)$ identifying $A$ with $A(0;r_1,r_2)$, the same coordinate identifies $A_1$ with an annulus of the form $A(0;r_1,r)$, with $r\in(r_1,r_2]$. If we write $\omega_{|A_2}=\sum_{i\in\Z}\alpha_i\,t^i\,dt$, then $\omega_{A_1}$ will have the same expansion, hence the proof.
\end{proof}
This result also shows that the following definition is correct.
\begin{defn}
 Let $W$ be a wide open curve, $\omega\in\Omega_W$ and $e\in \cE(W)$. We define the residue of $\omega$ at $e$ as 
 $$
 \res_{(W,e)}(\omega):=\res_{(A,e)}(\omega),
 $$
 where $A$ is any open annulus in $W$ that belongs to the end $e$.
 
 If $W$ is clear from the context, we may omit writing it in the index and just write $\res_e$.
\end{defn}

For the later purposes, we will need the following lemma.
\begin{lemma}
 Let $D$ be an open disc, $A$ its boundary annuli and $e\in \cE(A)$ which is different from the end $e'$ of $D$. Let $\omega$ be a meromorphic differential on $D$ having finitely many poles, say $a_1,\dots,a_n$, none of which belong to $A$ (that is, $\omega$ is a differential form on the wide open curve $D\setminus\{a_1,\dots,a_n\}$). Then,
 $$
 \sum_{i=1}^n\res_{a_i}(\omega)=\res_e(\omega_{|A}).
 $$
\end{lemma}
\begin{proof}
 Let $t$ be a coordinate on $D$ (hence also on $A$). By the $p$-adic Mittag-Leffler decomposition, $\omega$ can be written as
 $$
 \omega=f_0(t)\,dt+\sum_{i=1}^nf_i(t)\,dt,
 $$
 where $f_0$ is a holomorphic function on $D$ and each $f_i(t)$ is a holomorphic function on the annulus $D\setminus\{a_i\}$. In particular,
 $\res_e(\omega)=\sum_{i=1}^n\res_e(f_i(t)\,dt)$ (since residue is clearly additive). Finally, for each $i=1,\dots,n$, we have by Lemma \ref{lem: residue oposites} that $\res_{a_i}(f_i(t)\,dt)=-\res_{e'}(f_i(t)\,dt)=\res_e(f_i(t)\,dt)$.
\end{proof}

\section{Residue theorem}
The results in this section are more or less contained in \cite{Col89}. However, the proofs seem to be new and do not rely on the cohomological results as in \lc but rather on an approximation argument together with the classical results valid for $k$-algebraic curves.

\begin{thm}(Compare with \cite[Proposition 4.3.]{Col89})\label{thm: residue}
 Let $W$ be wide open curve and $\omega\in\Omega_W$. Then
 $$
 \sum_{e\in\cE(W)}\res_e(\omega)=0.
 $$
\end{thm}
\begin{proof}
Let $e_1,\dots,e_n$ be all the ends of $W$, and for each $i=1,\dots,n$, let us choose $A'_i, A''_i\in e_i$ with $A''_i\subsetneq A'_i$. Let $x_i$ be the endpoint of $A''_i$ in $A'_i$ and let $A_i$ be a boundary open annulus in $A_i'$ with an endpoint $x_i$ and which belongs to the other end $e_i'$ of $A_i'$. Let $X$ be the $k$-affinoid curve $W\setminus\cup_{i=1}^nA''_i$ (note that it is a finite curve which does not have ends, hence is compact, hence affinoid). Let $X'$ be a simple projectivization of $W$ (hence of $X$ as well), and let $D_1,\dots,D_n$ be open discs/connected components of $X'\setminus X$ indexed in such a way that $A''_i$ is a boundary annulus of $D_i$. Finally, for $i=1,\dots,n$ let $a_i\in D_i$ be a rational point. Let $\cX:=\cX'\setminus\{a_1,\dots,a_n\}$, where $\cX'$ is the smooth projective $k$-algebraic curve associated to $X'$. 

By $p$-adic Runge's theorem \ref{thm: runge}, the $\cO_{\cX}$-module $\Omega_{\cX}(\cX)$ of regular rational differential forms on $\cX$ is dense in $\Omega^\dagger_X$. Hence, there is a sequence $(\omega_m)_{m\in\N}$ of differential forms in $\Omega_{\cX}(\cX)$ such that $|\omega_m-\omega|_X<\frac{1}{m}$. In particular, $\sup_{x\in A_i}|\omega_m-\omega|_x<\frac{1}{m}$.

Let $t_i:A_i\to A(0;r_i,r_i')$ be an orientation preserving coordinate on $(A_i,f_i)$, where $f_i$ is the end of $A_i$ such that each annulus in $f_i$ is attached to $x_i$, and let us fix $\rho_i\in (r_i,r_i')$ and let $\rho:=\max\{\rho_1,\dots,\rho_n\}$. Let us put 
$$\omega_{|A_i}=\sum_{j\in\Z}\alpha_{i,j}\,t_i^j\,dt_i\quad \text{and }(\omega_m)_{|A_i}=\sum_{j\in\Z}\beta_{m,i,j}\,t_i^j\,dt_i,\quad i=1,\dots,n.
$$
By what we said above, $|\alpha_{i,-1}-\beta_{m,i,-1}|<\frac{1}{m}\,\rho_i$, so 
\begin{equation}\label{eq: first relation}
|\sum_{i=1}^n\alpha_{i,-1}-\sum_{i=1}^n\beta_{m,i,-1}|<\frac{\rho}{m},\quad m=1, 2\dots.
\end{equation}
 But, $\beta_{m,i,-1}=\res_{f_i}(\omega_m)_{|A_i}=-\res_{e'}(\omega_m)_{|A_i}=\res_{a_i}(\omega_m)_{|(D_i\cup A_i')\setminus\{a_i\}}=\res_{a_i}(\omega_m)$, where the last expression is the classical residue of $\omega_m$ at $a_i$ and we used Lemma \ref{lem: residue oposites} and the fact that $D_i\cup A_i'$ is an open disc, by construction. By the algebraic residue theorem, it follows that $\sum_{i=1}^n\beta_{m,i,-1}=0$ so the equation \eqref{eq: first relation} becomes
 \begin{equation}
  |\sum_{i=1}^n\alpha_{i,-1}|<\frac{1}{m}
 \end{equation}
Similarly, $\alpha_{i,-1}=\res_{f_i}(\omega_{|A_i})=-\res_{e_i'}(\omega_{|A_i})=\res_{e_i}(\omega)$. So, finally
$$
|\sum_{e\in\cE(W)}\res_e(\omega)|<\frac{1}{m}, \quad,m=1,2,\dots,
$$
 which finishes the proof.
\end{proof}
The following two corollaries can be seen as non-archimedean versions of the classical inside-outside residue theorems.
\begin{cor} Let $X$ be a finite curve, $X'$ its simple projectivization and $W$ some wide open curve in $X'$ that contains $X$. Let $\omega\in\Omega_W$ \footnote{In other words, let $\omega$ be an {\em overconvergent} differential form on $X$}. Then,
 $$
 \sum_{e\in \cE(X)}\res_e(\omega)=\sum_{e\in \cE^o(X)}\res_e(\omega)
 $$
\end{cor}
\begin{rmk}
 Since $\omega\in \Omega_W$, for each outer end $e\in\cE^0(X)$, there is an annulus $A_e\in e$ where $\omega$ is defined. Then, it makes sense to write $\res_e(\omega)$ as it does not depend on the chosen $A_e$.
\end{rmk}
\begin{proof}
 By Shrinking $W$ if necessary we may assume that $W\setminus X$ is a disjoint union of open annuli all of which are attached to the boundary of $X$. That is, we may assume that $W\setminus X$ is a disjoint union of annuli that belong to different outer ends of $X$, and each end has its representative. 
 
 Let $A_1,\dots,A_n$ be the connected components of $W\setminus X$ and suppose $A_i\in e_i$, where $e_i\in\cE^o(X)$, for $i=1,\dots,n$. Then, $e_i\in\cE(A_i)$ and let us denote by $e'_i$ the other end of $A_i$. The Theorem \ref{thm: residue} together with Lemma \ref{lem: residue oposites} implies
 \begin{align*}
 0&=\sum_{e\in \cE(W)}\res_e(\omega)=\sum_{e\in\cE(X)}\res_e(\omega)+\sum_{i=1}^n\res_{(A_i,e_i')}(\omega)=\sum_{e\in\cE(X)}\res_e(\omega)-\sum_{i=1}^n\res_{(A_i,e)}(\omega)\\
 &=\sum_{e\in\cE(X)}\res_e(\omega)-\sum_{e\in\cE^0(X)}\res_e(\omega).
 \end{align*}
\end{proof}

\begin{cor}
 Let $U,\,V$ and $W$ be wide open curves with $U\cup V=W$. Suppose further that each connected component of $U\cap V$ is an open annulus, and let $\omega\in\Omega_W$. Then,
 $$
 \sum_{e\in\cE(U)\cap \cE(W)}\res_e(\omega)=-\sum_{e\in\cE(V)\cap \cE(W)}\res_e(\omega).
 $$
\end{cor}
\begin{proof}
 It is enough to note that $\cE(U)\cap \cE(W)$ and $\cE(V)\cap \cE(W)$ are disjoint, that is, that $U$ and $V$ do not have any common ends under the conditions of the theorem. The result will then follow from Theorem \ref{thm: residue}
 
 If $U$ and $V$ do have a common end, say $e'$, it means that there is a common open annulus, say $A'$, that belongs to this end $e'$ and consequently there is an annulus $A$ which is a connected component of $U\cap V$ that contains $A'$ (and that is also in $e'$). The closure of $A$ in both $U$ and $V$ contains in addition only one point which is necessarily attached to the other end of $A$, hence must belong to both $U$ and $V$, which is a contradiction.
\end{proof}

%
%
%

\section{Mittag-Leffler problems}\label{s: M-L}

Let $W$ be a wide open curve, let $e_1,\dots,e_n$ be its ends and let $A_i$ be an open annulus in $e_i$. Let $X'$ be a simple projectivization of $W$, let $D'_1,\dots,D'_n$ be connected components of $X'\setminus W$ indexed so that the closure of $A_i$ in $X'$ intersects $D'_i$. Let $D_i$ be the open disc $A_i\cup D'_i$, and $a_i\in D'_i(k)$ a rational point. As usual, we will denote the corresponding $k$-algebraic curve for $X'$ by $\cX'$.

Now, for each $i=1,\dots,n$ let $t_i$ be a coordinate on the disc $D_i$ centered at $a_i$ (that is, $t_i(a_i)=0$) which identifies $A_i$ with $A(0;r_i,r_i')$. Let further $f_i:=f_i(t_i)=\sum_{j\leq j(i)}f_{i,j}\,t_i^j$ be a holomorphic function on $A_i$ and let $\omega_i=\sum_{j\leq j(i)}\alpha_{i,j}\,t_i^j\,dt_i$ be a differential form on $A_i$.

The {\em Mittag-Leffler problem for the data $\cD:=\{(A_i,f_i)\mid i=1,\dots,n\}$} is to find  a holomorphic function $f$ on $W$ whose restriction on $A_i$ is $f_i$.

The {\em Mittag-Leffler problem for the data $\cD^1:=\{(A_i,\omega_i)\mid i=1,\dots,n\}$} is to find a differential form $\omega$ on $W$ whose restriction on $A_i$ is $\omega_i$. 

\begin{rmk}\label{rmk: classic M-L}
 Suppose that each of the functions $f_i$ (resp. differential forms $\omega_i$) has only finitely many terms. Then, $f_i$ (resp. $\omega_i$) is convergent on the whole punctured open disc $D_i\setminus\{a_i\}$ and the Mittag-Leffler problem for the data $\cD$ (resp. $\cD^1$) is then just the classic Mittag-Leffler problem for meromorphic functions (resp. for meromorphic differentials) on $k$-algebraic curve $\cX'$.
\end{rmk}

To state our main results, let us denote by $\D$ the divisor $\sum_{i=1}^n(j(i)+1)\,[a_i]$ on $\cX'$ and by $\cL^{(1)}(\D)$ the space of meromorphic differentials $\omega$ on $\cX'$ which are holomorphic outside of the points $a_i$ and with $\ord_{a_i}(\omega)\geq -\ord_{a_i}(\D)$, for $i=1,\dots,n$.

\begin{thm}\label{thm: M-L funcitons}
 The following are equivalent:
 \begin{enumerate}
  \item The Mittag-Leffler problem for the data $\cD$ has a solution.
  \item For every differential form $\omega$ on $W$,
  $$
  \sum_{i=1}^n\res_{e_i}(f_i\,\omega)=0.
  $$
  \item For every differential form $\omega\in \cL^{(1)}(\D)$,
  $$
  \sum_{i=1}^n\res_{e_i}(f_i\,\omega)=0.
  $$
 \end{enumerate}
\end{thm}
\begin{thm}\label{thm: M-L differenitals}
 The Mittag-Leffler problem for the data $\cD^1$ has a solution if and only if 
 $$
 \sum_{i=1}^n\res_{e_i}(\omega_i)=0.
 $$
\end{thm}

\begin{rmks}\label{rmks: M-L}
{\em 1.} Continuing Remark \ref{rmk: classic M-L}, if the functions $f_i$ (resp. differentials $\omega_i$) have finitely many terms, then Theorem \ref{thm: M-L differenitals} and the equivalence of {\em 1.} and {\em 3.} are just the classical statements concerning the Mittag-Leffler problems for meromorphic differentials and meromorphic functions on smooth $k$-algebraic curves (see for example \V{to add a citation}).   \\
{\em 2.} If we consider for each $i=1,\dots,n$, a subannulus $A'_i:=A(0;r_{i,1},r_{i,2})$ of $A_i$, with $r_i\leq r_{i,1}<r_{i,2}\leq r_i'$ and we put $W':=W\setminus \cup_{i=1}^n A(0;r_i,r_{i,1}]$, then the Mittag-Leffler problem for the data $\cD':=\{(A_i',f_i)\mid i=1,\dots,n\}$ (resp. $\cD'^1:=\{(A_i,\omega_i)\mid i=1,\dots,n\}$) has a solution (on $W'$) if and only if the Mittag-Leffler problem for $\cD$ (resp. $\cD^1$) has a solution (on $W$). Indeed, any solution for the data $\cD$ (resp. $\cD^1$) on $W$ is naturally a solution for $\cD'$ (resp. $\cD'^1$) by simply restricting it to $W'$. In the other direction, a solution on $W'$ for $\cD'$ (resp. $\cD'^1$) extends naturally to the whole $W$ because function $f_i$ converges on the whole open disc $A_i$.\\
{\em 3.} If there exists a solution to our Mittag-Leffler problem, it is unique. For example, if $F_1$ and $F_2$ are solutions to Mittag-Leffler problem for the data $\cD$, then $F_1-F_2$ restricted to each open annulus $A_i$ is zero, hence is zero all over $W$.
\end{rmks}

The idea behind both proofs is to "approximate" functions $f_i$ (resp. differentials $\omega_i$) by a sequence of functions (resp. differentials) that have only finitely many terms and so that each member of the corresponding sequence of classical Mittag-Leffler problems admits a solution. Then, the sequence of these solutions will converge to the solution of our original problem.

Let $A_i':=A(0;r_1,r_i'']$, $r_i''\in(r_i,r_i')$ be a semi-open annulus in $A_i$ and let us put $A_i'':=A(0;r_i'',r_i')$ and $X:=W\setminus\cup_{i=1}^n$. Then $X$ is a $k$-affinoid curve in $W$ and in $X'$ and $X'\setminus X$ is a disjoint union of open discs. Furthermore, let $x_i$ be the point in the Shilov boundary of $X$ such that $x_i$ is in the closure of $A_i'$ in $W$ (in other words, $x_i$ is the Shilov point of the closed annulus $A[0;r_i'',r_i'']$). The norm $|\cdot|_X$ will serve us as to approximate. 

The proof of the Theorem \ref{thm: M-L differenitals} is simpler, so we start with it.

\begin{proof}[Proof of Theorem \ref{thm: M-L differenitals}]
 One direction is clear by Theorem \ref{thm: residue}. For the other, let $I:=\min\{j(i)\mid i=1,\dots,n\}$ and for $l<I$ let us put 
 $$
 \omega_{i,l}:=\sum_{s=l}^{j(i)}\alpha_{i,s}\, t_i^s\, dt_i,\quad i=1,\dots,n.
 $$
 Clearly, 
 $$
 \sum_{i=1}^n\res_{e_i}(\omega_i)=\sum_{i=1}^n\res_{e_i}(\omega_{i,l})=0,\quad l=I-1,I-2,\dots,
 $$
 so, by the classical Mittag-Leffler problem for differentials there exists a differential form $\sigma_l$ on $\cX'$, holomorphic outside of $\{a_1,\dots,a_n\}$ and whose local expansion at $a_i$ is precisely $\omega_{i,l}$. By construction 
 $$
 |\sigma_l-\omega_i|_{x_i}=\sup_{s<l}\{|\alpha_{i,l}|\,(r_{i}'')^s\},
 $$
 which tends to 0 as $l$ goes to infinity. Consequently, $(\sigma_l)_{l <I}$ is a Cauchy sequence for the $|\cdot|_X$ norm, and there is a limit $\sigma:=\lim\limits_{l\to-\infty}\sigma_l\in\Omega_X$. The form $\sigma$ restricted to the open annulus $A_i''$ coincides with $\omega_i$ by construction of the sequence $\sigma_l$, hence is a solution to the Mittag-Leffler problem for the data $\{(A_i'',f_i)\mid i=1,\dots,n\}$. By Remark \ref{rmks: M-L} {\em 2.} $\sigma$ is the solution for $\cD^1$.
\end{proof}

\subsubsection{}\label{ss: 1}
The following sections contain preliminary results that will be used for the proof of Theorem \ref{thm: M-L funcitons}. The idea is to truncate functions $f_i$ and modify them so that the corresponding classic Mittag-Leffler problem has a solution. The limit of these solutions will be the solution of our problem.\\

Let $\sigma_1,\dots,\sigma_m$ be the $k$-basis for $\cL^{(1)}(\D)$ and let $I:=\max\{j(i)\mid i=1,\dots,n\}$ and $I'=\min\{j(i)\mid i=1,\dots,n\}$. For each $j=1,\dots,m$ we write
$$
\sigma^{(i)}_j:=\sigma_{j|A_i}=\sum_{l=-\infty}^{-j(i)-1}\alpha^{(i)}_{j,l}\,t_i^l\,dt_i=\sigma_{j|A_i}=\sum_{l=-\infty}^{-I-1}\alpha^{(i)}_{j,l}\,t_i^l\,dt_i,\quad i=1,\dots,n,
$$
(agreeing that some of the top terms in the last series may be 0) and we note that the condition 
$$
\sum_{i=1}^n\res_{e_i}(f_i\,\sigma_j)=\sum_{i=1}^n\res_{e_i}(f_i\, \sigma^{(i)}_j)=0,\quad j=1,\dots,m,
$$
is equivalent to 
\begin{equation}\label{eq: new conditions}
\sum_{i=1}^n\sum_{l=-\infty}^{I}f_{i,l}\,\alpha^{(i)}_{j,-l-1}=0,\quad j=1,\dots,m.
\end{equation}
Now, for $l<I'$, let us put 
$$
\vphi_{i,l}:=\vphi_{i,l}(t_i)=\sum_{s=l}^{I}f_{i,s}\,t^s_i=\sum_{s=l}^{j(i)}f_{i,s}\,t^s_i,\quad i=1,\dots,n,
$$
and
$$
\beta_{j,l}:=\sum_{i=1}^n\sum_{s=-\infty}^{l-1}f_{i,s}\,\alpha^{(i)}_{j,-s-1},\quad j=1,\dots,m.
$$
In particular, equation \eqref{eq: new conditions} implies that 
\begin{equation}\label{eq: cons 1}
\sum_{i=1}^n\res_{e_i}(\vphi_{i,l}\,\sigma_j)=\sum_{i=1}^n\sum_{s=l}^{I}f_{i,s}\,\alpha^{(i)}_{j,-s-1}=-\beta(j,l),\quad j=1,\dots,m.
\end{equation}
$$
\lim_{l\to -\infty}\beta(j,l)=0, \quad j=1,\dots,m.
$$
If we put $B_l:=\max\{|\beta_{j,l}|\mid j=1,\dots,m\}$, then the last equation is equivalent to
\begin{equation}\label{eq: cons lim}
\lim\limits_{l\to \infty}B_l=0.
\end{equation}
\subsubsection{}\label{ss: 2}

For $l<I'$ and for $j=1,\dots,m$, let us put
$$
A_{j,l}:=(\alpha^{(1)}_{j,-I-1},\alpha^{(2)}_{j,-I-1},\dots,\alpha^{(n)}_{j,-I-1},\dots,\alpha^{(1)}_{j,-I-2},\dots,\alpha^{(n)}_{j,-l}),
$$
and 
$$
A_{j,\infty}:=(\alpha^{(1)}_{j,-I-1},\alpha^{(2)}_{j,-I-1},\dots,\alpha^{(n)}_{j,-I-1},\dots,\alpha^{(1)}_{j,-I-2},\dots,\alpha^{(n)}_{j,-l},...)\in k^{\N}.
$$
\begin{lemma}
 If for all $l\leq I'$ the vectors $A_{1,l},\dots,A_{m,l}$ are linearly dependent over $k$, then so are the vectors $A_{1,\infty},\dots,A_{m,\infty}$.
\end{lemma}
\begin{proof}
 For $l<I'$, let $B_{l}:=(b_{1,l},\dots,b_{m,l})\in k^m\setminus\{\vec{0}\}$ such that $\sum_{j=1}^m b_{j,l}\, A_{j,l}=\vec{0}$. 
 We note that for each natural number $s$ we also have
 \begin{equation}\label{eq: l+s}
  \sum_{j=1}^m b_{j,l-s}\, A_{j,l}=\vec{0}.
 \end{equation}
We distinguish the following two possibilities:\\
{\emph {1. There exists a strictly increasing sequence of natural numbers $(s_i)_{i\in\N}$ such that $B_{l+s_i}$ is proportional to $B_{l}$}}.\\
In this case, having in mind \eqref{eq: l+s}, we see that 
$$
\sum_{j=1}^mb_{j,l}\,A_{j,\infty}=\vec{0},
$$
and the lemma follows.\\
{\emph {2. No such a sequence exists.}}\\
In this case we argue by induction on me and first notice that if $m=1$ the lemma is trivial. Suppose that $m>1$. We claim that we can find a number $I_1\leq I'$, and a sequence of vectors $B^{(1)}_l:=(b^{(1)}_{1,l},\dots,b^{(1)}_{m-1,l})\in k^{m-1}\setminus{\vec{0}}$, $l\leq I_1$, and such that 
$$
\sum_{j=1}^{m-1}b^{(1)}_{j,l}\, A_{j,l}=\vec{0},\quad,l=I_1,I_1-1,\dots.
$$
Indeed, suppose that $l\leq I'$ is such that $b_{m,l}\neq 0$. Then, either for all natural numbers $s>0$, $b_{m,l-s}=0$, in which case we can take $I_1=l-1$, and $b^{(1)}_{j,l}:=b_{j,l}$ for $j=1,\dots,m-1$, either there exists some $s_0>0$ such that $b_{j,l-s_0}\neq 0$ and $B_{l-s_0}$ is not proportional to $B_l$. In this case we may take $I_1=l$ and $b^{(1)}_{j,l}:=b_{j,l}-\frac{b_{m,l}\, b_{j,l-s_0}}{b_{m,l-s_0}}$, $j=1,\dots,m-1$, because of \eqref{eq: l+s}, and because $B_{l-s_0}$ is not proportional to $B_l$ we conclude that $B^{(1)}_l\neq \vec{0}$. 

Now we have that for all $l\leq I_1$, the vectors $A_{1,l},\dots,A_{m-1,l}$ are linearly dependent over $k$ so by inductive hypothesis, the vectors $A_{1,\infty},\dots,A_{m-1,\infty}$ are linearly dependent as well and the lemma is proved.
\end{proof}
\begin{lemma}
 The vectors $A_{1,\infty},\dots,A_{m,\infty}$ are linearly independent over $k$.
\end{lemma}
\begin{proof}
 Suppose that for some $b_1,\dots,b_m\in k$, we have $b_1\,A_{1,\infty}+\dots+b_m\,A_{m,\infty}=\vec{0}$. By our construction of the vectors $A_{j,\infty}$ it follows that for each $i=1,\dots,n$,
 $$
 b_1\,\sigma^{(i)}_1+\dots+b_m\, \sigma^{(i)}_m=0.
 $$
 But then, $\sigma:=b_1\, \sigma_1+\dots+b_m\,\sigma_m$ restricted to each annulus $A_i$ is the zero differential form, hence $\sigma$ itself is zero. This implies that $b_1=\dots=b_m=0$.
\end{proof}
An immediate consequence of the previous two lemmas is the following.
\begin{cor}\label{cor: consequence}
 For some $I_2\leq I'$, the vectors $A_{1,l},\dots,A_{m,l}$ are linearly independent over $k$ for all $l\leq I_2$.
\end{cor}

\subsubsection{}\label{ss: 3}

Let us denote by
$$
\cM_{\infty}:=\left(\begin{array}{c}
               A_{1,\infty}\\
               \vdots\\
               A_{m,\infty}
              \end{array}\right)
$$
the $m\times \infty$ matrix whose rows are the vectors $A_{j,\infty}$, constructed above. Further, let us fix natural numbers $s_1<\dots<s_m$ such that the $m\times m$ matrix $\cM$ whose $j^{th}$ column is equal to the $s_j^{th}$ of $\cM_{\infty}$, $j=1,\dots,m$, and such that $\cM$ is invertible (we can find such numbers because of Corollary \ref{cor: consequence}). 

In other words, we fix $i_1,\dots,i_m\in\{1,\dots,n\}$ and integers $l_1<\dots<l_m$ such that 
$$
\cM=\left(\begin{array}{ccc}
           \alpha^{i_1}_{1,-l_1-1}&\dots&\alpha^{i_m}_{1,-l_m-1}\\
           \vdots&\dots&\vdots\\
           \alpha^{i_1}_{m,-l_1-1}&\dots&\alpha^{i_m}_{m,-l_m-1}
          \end{array}
\right).
$$

Let $M=|\det(\cM)|$ and $M'$ the maximum among the norms of all $(m-1)\times(m-1)$ minors of $\cM$. 

Let $X_{i_1,-l_1-1}=x_{i_1,-l_1-1},\dots,X_{i_m,-l_m-1}=x_{i_m,-l_m-1}$ be the solution of the system 
\begin{equation}\label{eq: sys 2}
\cM\, \left(\begin{array}{c}X_{i_1,-l_1-1}\\ \vdots\\ X_{i_m,-l_m-1}\end{array}\right)=\left(\begin{array}{c}\beta_{1,l_m}\\ \vdots\\ \beta_{m,l_m}\end{array}\right),
\end{equation}
\begin{cor} We have $|x_{i_s,-l_s-1}|\leq B_{l_m}\, \frac{M'}{M}$, $s=1,\dots,m$. 
\end{cor}
\begin{proof}
 This is a simple consequence of Cramer's rule and development of the determinant via minors.
\end{proof}

Let $l<l_m$ be an integer, and let us define for $i=1,\dots,n$,
$$
\phi_{i,l}:=\phi_{i,l}(t_i)=\vphi_{i,l}+\sum_{\substack{s=1\\i_s=i}}^mx_{i_s,-l_i-1}\, t_i^{l_i}=\sum_{s=l}^If_{i,s}\, t^s_i+\sum_{\substack{s=1\\i_s=i}}^mx_{i_s,-l_i-1}\, t_i^{l_i}.
$$
The functions$\phi_{i,l}$ are obviously Laurent polynomials.
\begin{lemma}\label{lem: phi_l}
 The classic Mittag-Leffler problem for the data $\{(A_i,\phi_{i,l})\mid i=1,\dots,n\}$ (see Remarks \ref{rmk: classic M-L} and \ref{rmks: M-L} {\em {1.}}) admits a solution.
\end{lemma}
\begin{proof}
 This amounts to the following calculations, for $j=1,\dots,m$,
 \begin{align*}
 \sum_{i=1}^n\res_{e_i}(\phi_{i,l}\,\sigma_j)&=\sum_{i=1}^n\res_{e_i}(\vphi_{i,l}\, \sigma_j)+\sum_{i=1}^n\res_{e_i}\Big(\sum_{\substack{s=1\\i_s=i}}^mx_{i_s,-l_i-1}\, t_i^{l_i}\, \sigma_j\Big)\\
 &=-\beta_{j,l}+\sum_{i=1}^n\sum_{\substack{s=1\\i_s=i}}^mx_{i_s,-l_i-1}\, \alpha^{(i_s)}_{j,-l_i-1}\quad\text{(by \eqref{eq: cons 1})}\\
 &=-\beta_{j,l}+\sum_{s=1}^mx_{i_s,-l_i-1}\, \alpha^{(i_s)}_{j,-l_i-1}\\
 &=-\beta_{j,l}+\beta_{j,l}=0.
 \end{align*}
\end{proof}
\begin{lemma}
 Let $\phi_l$ be the solution of the classic Mittag-Leffler problem for the data $\{(A_i,\phi_{i,l})\mid i=1,\dots,n\}$, where $l<l_m$. Then, the sequence $(\phi_l)_{l<l_m}$ is Cauchy for the norm $|\cdot|_X$. 
\end{lemma}
\begin{proof}
 For an integer $l<l_m$ and $s$ a natural number we have
 \begin{align*}
 |\phi_{l}-\phi_{l-s}|_X&=\max\{|\phi_l-\phi_{l-s}|_{x_i}\mid i=1,\dots,n\}\\
 &=\max\{|\phi_{i,l}-\phi_{i,l-s}|_{x_i}\mid i=1,\dots,n\}\\
 &=\max\{\max\limits_{s=l-s,\dots,l-1}|f_{i,s}|\,(r_i'')^s\mid i=1,\dots,n\},
 \end{align*}
 and the last expression tends to 0 as $l$ tends to $-\infty$.
\end{proof}

We may finally start with the proof of Theorem \ref{thm: M-L funcitons}.
\begin{proof}[Proof of Theorem \ref{thm: M-L funcitons}] If $f$ is the solution of the Mittag-Leffler problem for $\cD$, then {\emph{ 2.}} is true by Theorem \ref{thm: residue}. The implication \emph{ 2.} to \emph{ 3.} is clear, so we only need to prove that \emph{3.} implies \emph{1.}

Let $\phi:=\lim\limits_{l\to-\infty}\phi_l$. Then $\phi\in \cA_X$ and its restriction on the annulus $A_i''$ coincides with $f_i$, by the construction of the functions $\phi_{i,l}$. Hence, by Remark \ref{rmks: M-L} {\emph{ 2.}} $\phi$ extends to an analytic function $W$ and is a solution to the Mittag-Leffler problem for $\cD$. 
\end{proof}

\begin{cor}
 Let $X$ be a semi-open finite curve, let $e_1,\dots,e_{n'}$ be its ends, and for $i=1,\dots,n'$, let $A_i\in e_i$ be disjoint open annuli and $t_i:A_i\to A(0;r_i,r_i')$ an orientation preserving coordinate on the oriented annulus $(A_i,e_i)$. 
 
 Suppose that for every $i=1,\dots,n'$, we are given an analytic function (resp. a differential form) $f_i:=f_i(t_i)=\sum_{j=-\infty}^{j(i)}f_{i,j}\,t_i^j$ (resp. $\omega_i=\sum_{j=-\infty}^{j(i)}\alpha_{i,j}\, t_i^j\, dt_i$) on $A_i$. 
 
 Then, there exists a function $f$ (resp. a differential form $\omega$) on $X$ such that $f_{|A_i}=f_i$ (resp. $\omega_{|A_i}=\omega_i$).
\end{cor}
\begin{proof}
Let $X'$ be a simple projectivization of $X$ and let $W$ be a wide open curve in $X'$ that contains $X$ and such that $\cE(X)\subset \cE(W)$. Let us further denote by $e_{n'+1},\dots,e_{n}$ the ends of $\cE(W)$ that are not in $\cE(X)$, and for each $i=n'+1,\dots,m$ let $A_i$ be an open annulus in $e_i$ (disjoint from all the annuli involved) and $t_i:A\to A(0;r_i,r_i')$ an orientation preserving coordinate on $(A_i,e)$.

 We first deal with the statement for the analytic functions, where the idea is to find an analytic function $f_i:=f_i(t_i)=\sum_{-\infty}^{j(i)}f_{i,j}\,t_i^j$ on an open annulus $A_i$ such that the Mittag-Leffler problem for the data $\cD=\{(A_i,f_i)\mid i=1,\dots,m\}$ has a solution $f$ on $W$. Naturally, function $f$ will satisfy the conditions from the corollary.
 
 Using the notation from Sections \ref{ss: 1}, \ref{ss: 2} and \ref{ss: 3} we may write the equation \eqref{eq: new conditions} as
 \begin{align}
  0&=\sum_{i=1}^n\sum_{l=-\infty}^If_{i,l}\,\alpha^{(i)}_{j,-l-1}=\sum_{i=1}^{n'}\sum_{l=-\infty}^If_{i,l}\,\alpha^{(i)}_{j,-l-1}+\sum_{i=n'+1}^n\sum_{l=-\infty}^If_{i,l}\,\alpha^{(i)}_{j,-l-1}\nonumber\\
  &=S_j+\sum_{i=n'+1}^n\sum_{l=-\infty}^If_{i,l}\,\alpha^{(i)}_{j,-l-1},\quad j=1,\dots,m.\label{eq: new system}
 \end{align}
By Corollary \ref{cor: consequence}, for sufficiently small integer $l$ vectors $A_{1,l},\dots,A_{m,l}$ are linearily independent over $k$, and consequently, and similarly to what happened in Section \ref{ss: 3}, the system \eqref{eq: new system} admits a solution where for each $i=n'+1,\dots,n$, there are only finitely many nonzero $f_{i,j}$, so that in particular $f_i$ is an analytic function on $A_i$. Then, Mittag-Leffler problem for data $\cD$ admits a solution $f$ on $W$ by Theorem \ref{thm: M-L funcitons}, and this proves the statement.

The question concerning differential forms is slightly easier. Namely, for each $i=n'+1,\dots,n$, we choose a differential form $\omega_i$ on $A_i$ such that the sum of the residues $\sum_{i=1}^n\res_{e_i}(\omega_i)=0$ (basically, we can choose arbitrary $\omega_i$ for $i=n'+1,\dots, n-1$ and $\omega_n:=a/t_n\,dt_n $ where $a=-\sum_{i=1}^{n-1}\res_{e_i}(\omega_i)$). By Theorem \ref{thm: M-L differenitals} there exists a differential form $\omega$ on $W$ whose restriction on $A_i$ is $\omega_i$. 
\end{proof}

\begin{rmk}
 It is most likely that Theorems \ref{thm: M-L differenitals} and \ref{thm: M-L funcitons} can be extracted from the global duality results as in \cite[Sections 7, 8, 9]{CrewFinite} or \cite{vanDerPutDuality}. However, one will have to pay the price of the difficult results involved.
\end{rmk}

\bibliographystyle{plain}
\bibliography{biblio}
\end{document}